\documentclass[11pt]{amsart}
\usepackage[utf8]{inputenc}
\usepackage{amsmath,amssymb}
\usepackage{wrapfig}
\usepackage{url}
\usepackage{indentfirst}
\usepackage{mathtools}
\usepackage{graphicx}
\usepackage{stmaryrd}
\usepackage{amsthm}
\usepackage{geometry}
\usepackage[colorlinks=true,linkcolor=blue,citecolor=blue]{hyperref}
\title{Exponential Taylor Series}
\author{André Pedroso Kowacs}
\date{2022}

\allowdisplaybreaks
\theoremstyle{definition}
\newtheorem{theorem}{Theorem}
\newtheorem{corollary}{Corollary}

\newtheorem{exmp}{Example}

\newtheorem{lemma}{Lemma}
\theoremstyle{definition}
\newtheorem{definition}{Definition}

\begin{document}

\maketitle
\section{Abstract}
This paper derives a way to express differentiable complex-valued functions as the sum of powers of $(1-e^{\lambda x})$, where $\lambda\in\mathbb{R}$, with an explicit formula for the remainder. This formulation is then used to associate an infinite series to $C^\infty$ functions, which is shown to recover the original function under suitable conditions on the remainder. These results are also used to calculate some infinite series involving Stirling Numbers, as well as providing a few examples.
\section{Introduction}
This paper was inspired by the works of Michael Ruzhansky and Ville Turunen on \cite{cite1}, where they present a similar exponential Taylor series expansion, though only for the case $\lambda=2\pi i$. Ville Turunen also presented this in \cite{cite2}, where he refers to \cite{cite3}. This paper generalizes this ``exponential" Taylor series and presents a new formula for the remainder term, which allows the study of convergence of the infinite series for $C^\infty$ functions. This raises questions if one could develop a theory similar to that of analytic functions based on these Taylor expansions. It is worth noting that this exponential Taylor expansion allows us to approximate periodic functions by periodic polynomials, by taking $\lambda = 2\pi i/T$, so that the error term is also periodic, which may be useful in some situations. This paper is organized as follows: First, we introduce the main notation used throughout this paper. Then we prove our main results for functions of a single variable and present a few examples. We then conclude by proving our main result for multivariable functions.

\section{Initial Definitions and Notations}
\begin{definition}\label{def1}
Let $\lambda \in \mathbb{C}\backslash\{0\}$. We recursively define the differential operators $D^{\lambda,(j)}_x$ by $D^{\lambda,(0)}_x= \text{Id}$ and 
$$D^{\lambda,(j+1)}_x := \left(\frac{1}{\lambda} \frac{d}{dx} - j\cdot \text{Id}\right)\circ D^{\lambda,(j)}_x,$$
for every $j\in\mathbb{N}$. We shall sometimes omit the subindex in this differential operator whenever the variable under differentiation is clear.
For $k\in\mathbb{N}$, let $C^k(U)$ be the set $k$ times continuously differentiable functions from and open subset $U\subset\mathbb{R}^n$ to $\mathbb{C}$. Also, for $x,y\in\mathbb{R}^n$, let $[[x,y]]$ be the open line segment from $x$ to $y$, that is:
$$[[x,y]] = \{z\in\mathbb{R}^n : z = (1-t)x+ty,\,t\in (0,1)\}.$$
\end{definition}
\section{Main results and theorems}

\begin{theorem}\label{teo1}

Let $I\subset \mathbb{R}$ be an open interval, $a\in C^k(I)$, $\lambda\in\mathbb{C}\backslash\{0\}$. Then,
$$a(x) = \sum_{j=0}^{N-1}\frac{1}{j!}D^{\lambda,(j)}a(x_0)(e^{\lambda (x-x_0)}-1)^j + R_N(x,x_0),$$
for all $N\leq k$ and $x\in I$, where
$$R_N(x,x_0) = \frac{\lambda}{(N-1)!}\int_0^1D^{\lambda,(N)}a((1-\theta)x_0+\theta x)(e^{\lambda(1-\theta)(x-x_0)}-1)^{N-1}(x-x_0)d\theta.$$
In particular, 
\begin{align*}
|R_N(x,x_0)|&\leq |\lambda|\frac{N}{(N!}\sup_{\xi \in[[x_0,x]]}|D^{\lambda,(N)}a(\xi)(e^{\lambda (x-\xi)}-1)^{N-1}\|x-x_0|\\
&\leq \frac{|\lambda|}{(N-1)!}\sup_{\xi \in[[x_0,x]]}|D^{\lambda,(N)}a(\xi)|\sup_{-|x-x_0|\leq\eta \leq |x-x_0|}|e^{\lambda \eta}-1|^{N-1}\|x-x_0|.
\end{align*}
\end{theorem}
\begin{proof}
Let 
$$F(t) = \sum_{j=0}^{N-1}\frac{1}{j!}D^{\lambda,(j)}a(t)(e^{\lambda (x-t)}-1)^j.$$ Then it is easy to see that
\begin{align*}
\frac{1}{\lambda}\frac{d}{dt}F(t) &=\\
	&=\sum_{j=0}^{N-1}\frac{1}{j!}\left\{\left[\left(\frac{1}{\lambda}\frac{d}{dt}-j+j\right)D^{\lambda,(j)}a(t)\right](e^{\lambda(x-t)}-1)^j\right.\\
 &\left.+D^{\lambda,(j)}a(t)\left[\frac{1}{\lambda}\frac{d}{dt}(e^{\lambda(x-t)}-1)^j\right]\right\}\\
	&=\sum_{j=1}^{N-1}\frac{1}{j!}\left\{\left[\vphantom{e^{\lambda}}\left(D^{\lambda,(j+1)}a(t)+jD^{\lambda,(j)}a(t)\right)\right](e^{\lambda(x-t)}-1)^j+\right. \\
 &\left.+D^{\lambda,(j)}a(t)\left[-j(e^{\lambda(x-t)}-1)^{j-1}((e^{\lambda(x-t)}-1)-e^{\lambda(x-t)})\right]\right\}+D^{(\lambda,1)}a(t)\\
 &=\sum_{j=1}^{N-1}\frac{1}{j!}\left[D^{\lambda,(j+1)}a(t)(e^{\lambda(x-t)}-1)^j+D^{\lambda,(j+1)}a(t)(e^{\lambda(x-t)}-1)^j\right]\\
 &+\frac{1}{(j-1)!}D^{\lambda,(j)}a(t)(e^{\lambda(x-t)}-1)^{j-1}((e^{\lambda(x-t)}-1)-e^{\lambda(x-t)})+D^{(\lambda,1)}a(t)\\
 &=\sum_{j=1}^{N-1}\left[\frac{1}{j!}D^{\lambda,(j+1)}a(t)(e^{\lambda(x-t)}-1)^j-\frac{1}{(j-1)!}D^{\lambda,(j)}a(t)(e^{\lambda(x-t)}-1)^{j-1}\right]\\
 &+\frac{1}{0!}D^{\lambda,(1)}a(t)(e^{\lambda(x-t)}-1)^0\\
&=\frac{1}{(N-1)!}D^{\lambda,(N)}a(t)(e^{\lambda(x-t)}-1)^{N-1},
\end{align*}
where on the last line we used that the sum above it is a telescopic sum. 
The result then follows from applying the fundamental theorem of calculus to $F$, namely:
$$F(x)-F(x_0) = \int_{x_0}^x \lambda\frac{1}{\lambda}\frac{d}{dt}F(t)dt, $$
and then performing the change of variables $t = (1-\theta)x_0+\theta x$.
\end{proof}

Notice that by taking $\lambda = \frac{2\pi i}{T}$, for some $T>0$, the sum in the expansion above is $T$-periodic, so that if $a$ is also $T$-periodic the remainder term $R_N(x,x_0)$ is also $T$-periodic. Moreover, note that in this case $|e^{\lambda \eta}-1|\leq 2$, for all $\eta\in\mathbb{R}$, and, in particular, $|e^{\lambda \eta}-1|<1$ for $|\eta|<\frac{|T|}{6}$.

\begin{corollary}\label{coro}
Let $0\neq \lambda\in\mathbb{C}$. Let $x_0,x\in\mathbb{R}$ and $a\in C^\infty(\mathbb{R})$ such that
$$\frac{1}{(N-1)!}\sup_{\xi \in[[x_0,x]]}|D^{\lambda,(N)}a(\xi)(e^{\lambda (x-\xi)}-1)^{N-1}|\to 0,$$
as $N\to \infty$. Then 
$$a(x) = \sum_{j=0}^{\infty}\frac{1}{j!}D^{\lambda,(j)}a(x_0)(e^{\lambda (x-x_0)}-1)^j.$$
In particular, if $a$ is $T$-periodic, $\lambda = \frac{2\pi i}{T}$, $T>0$ and
$$\sup_{\xi \in[0,T]}|D^{\lambda,(N)}a(\xi)|\leq C_0 (N+k)!,$$
for some $C_0>0,\,k\in\mathbb{N}$, and every $N\in\mathbb{N}$, then the series above converges uniformly to $a$ in $[x_0-T/6,x_0+T/6]\subset \mathbb{R}$.

\end{corollary}

This motivates the following definition:
\begin{definition}
Let $a\in C^\infty(\mathbb{R}),\,0\neq\lambda\in\mathbb{C}$ and $x_0\in\mathbb{R}$. We define its $\lambda$-exponential Taylor series centered at $x_0$ as the formal series:
$$S_{a,\lambda,x_0}(x) = \sum_{j=0}^{\infty}\frac{1}{j!}D^{\lambda,(j)}a(x_0)(e^{\lambda (x-x_0)}-1)^j.$$
\end{definition}

\begin{theorem}
For $a\in\ C^\infty(\mathbb{R}),\,0\neq\lambda\in\mathbb{C}$ and $x_0\in\mathbb{R}$, the $\lambda$-exponential Taylor series of $a$ centered at $x_0$ converges for all $x$ such that
$$|e^{\lambda(x-x_0)}-1|<\limsup_{j\to\infty}\frac{(j+1)|D^{\lambda,(j)}a(x_0)|}{|D^{\lambda,(j+1)}a(x_0)|}.$$
In particular, if $\lambda = \frac{2\pi i}{T}$, $T>0$, then it converges for all $x$ such that 
$$|x-x_0|<\frac{T\arcsin\left(\frac{r_{x_0}}{2}\right)}{\pi},$$
where
$$ r_{x_0} = \limsup_{j\to\infty}\frac{(j+1)|D^{\lambda,(j)}a(x_0)|}{|D^{\lambda,(j+1)}a(x_0)|},$$
if $r_{x_0}\leq 2$, or for all $x$ otherwise.
\end{theorem}
\begin{proof}
The proof follows directly from the ratio test for series, that is, the series $\sum_{n\in\mathbb{N}}a_n$ converges absolutely if $\limsup\frac{|a_{n+1}|}{|a_{n}|}<1$. The second claim follows from Euler's identity for the complex exponential function.
\end{proof}

\begin{exmp}[$1$-periodic series for the cosine] Let $a(x) = \cos(2\pi x)$, $\lambda = 2\pi i$ and $x_0= 0$. Then a simple calculation yields:
\begin{align}\label{cos}
D^{2\pi i,(j)}\cos(2\pi x) = \begin{cases}\frac{(-1)^j}{2} j!e^{-2\pi i x},\,j\geq 2\\
i\sin(2\pi x),\,j=1\\
\cos(2\pi x), j=0,
\end{cases}
\end{align}
so that 
\begin{equation}\label{cosseries}
    \cos(2\pi x) = 1+\sum_{j=2}^\infty\frac{(-1)^j}{2}(e^{2\pi i x}-1)^j,
\end{equation}
for $|x|< 1/6$. Indeed we will prove \eqref{cos} by induction. The cases $j=0$ and $j=1$ can be easily verified by direct calculation. Suppose now that the formula holds for some $j\geq 2$. Then
\begin{align*}
    D^{2\pi i,(j+1)}\cos(2\pi x)&=\left(\frac{1}{2\pi i}\frac{d}{dx}-j\cdot\text{Id}\right)D^{2\pi i,(j)}\cos(2\pi x)\\
    &=\left(\frac{1}{2\pi i}\frac{d}{dx}-j\cdot\text{Id}\right)\frac{(-1)^j}{2} j!e^{-2\pi i x}\\
    &=\frac{(-1)^{j+1}}{2} j!e^{-2\pi i x}+j\frac{(-1)^{j+1}}{2} j!e^{-2\pi i x}\\
    &=\frac{(-1)^{j+1}}{2} (j+1)!e^{-2\pi i x},
\end{align*}
proving that \eqref{cos} holds in this case also, and therefore for every $j\in\mathbb{N}$. The identity \eqref{cosseries} then follows from Corollary \ref{coro}.
\end{exmp}

\begin{exmp}
Let $a(x) = \log(1-e^{i x})$, $\lambda = i$ and $x_0= \pi$. Then
\begin{align}\label{log}
D^{i,(j)}_x\log(1-e^{ix}) = \begin{cases}-(j-1)!\dfrac{e^{ijx}}{(1-e^{ix})^{j}},\,j\geq 1\\
\log(1-e^{i x}),\,\qquad\quad j=0,
\end{cases}
\end{align}
so that 
\begin{equation}\label{logseries}
    \log(1-e^{i x}) = \log(2)+\sum_{j=1}^\infty\frac{(-1)^{j+1}}{j\cdot2^j}(e^{i(x-\pi)}-1)^j,
\end{equation}
for $|x-\pi|<\frac{1}{3}$. Formula \eqref{log} can be proved by induction similarly to the previous example, and so \eqref{logseries} follows from Corollary \ref{coro} once again.
\end{exmp}

\begin{lemma} For all $n\in\mathbb{N}$, we have that
\begin{equation}\label{stirling}
    D^{\lambda,(n)}_t=\sum_{j=0}^{n}S^{(j)}_{(n)}\frac{1}{\lambda^j}\frac{d^j}{(dt^{j})},
\end{equation}
where $S^{(j)}_{(n)}$ is the $(n,k)$-th Stirling number of first kind, that is, $x(x-1)...(x-n+1) = \sum_{k=0}^nS^{(k)}_{(n)}x^k$.
\end{lemma}
\begin{proof}
We prove this by mathematical induction. Identity \eqref{stirling} is clear in the case $n=1$. Suppose the formula above holds for some $n\in\mathbb{N}$. Then 
\begin{align*}
    D^{\lambda,(n+1)}&=\left(\frac{1}{\lambda}\frac{d}{dt}-n\cdot\text{Id}\right) D^{\lambda,(n)}\\
    &=\left(\frac{1}{\lambda}\frac{d}{dt}-n\cdot\text{Id}\right)\sum_{j=0}^{n}S^{(j)}_{(n)}\frac{1}{\lambda^j}\frac{d^j}{(dt^j)}\\
    &=\sum_{j=0}^{n}S^{(j)}_{(n)}\frac{1}{\lambda^{j+1}}\frac{d^{j+1}}{(dt^{j+1})}-\sum_{j=0}^{n}nS^{(j)}_{(n)}\frac{1}{\lambda^j}\frac{d^j}{(dt^{j})}\\
    &=S_{(n)}^{(n)}\frac{1}{\lambda^{n+1}}\frac{d^{n+1}}{dt^{n+1}}-nS_{(n)}^{(0)}\text{Id}+\sum_{j=1}^{n}\left(S^{(j-1)}_{(n)}-nS^{(j)}_{(n)}\right)\frac{1}{\lambda^{j}}\frac{d^{j}}{(dt^{j})}\\
    &=\left(S_{(n)}^{(n)}-nS^{(n+1)}_{(n)}\right)\frac{1}{\lambda^{n+1}}\frac{d^{n+1}}{dt^{n+1}}+\sum_{j=1}^{n}\left(S^{(j-1)}_{(n)}-nS^{(j)}_{(n)}\right)+S^{(0)}_{(n+1)}\text{Id}\\
    &=\sum_{j=0}^{n+1}S^{(j)}_{(n+1)}\frac{1}{\lambda^{j}}\frac{d^{j}}{(dt^{j})},
\end{align*}
where in the last identity we used the fact that $n\geq 1$ and so $S^{(0)}_(n)=S^{(0)}_{n+1}=0$. This proves the identity \eqref{stirling} hold for $n+1$ and therefore for every $n\in\mathbb{N}$.
\end{proof}

\begin{exmp} Let $a(x) = x$, $\lambda = 2\pi i$ and $x_0= 0$. Then by the previous lemma, a simple calculation yields:
\begin{align}
D^{2\pi i,(j)}_xx = \begin{cases}x ,\,\qquad\qquad \qquad \quad j=0\\
\frac{1}{2\pi i}(-1)^{j-1}(j-1)!,\,j\geq 1,\label{diffx}
\end{cases}
\end{align}
so that 
$$x = \sum_{j=1}^\infty\frac{(-1)^{(j-1)}}{j2\pi i}(e^{2\pi i x}-1)^j,$$
for $|x|\leq 1/6$. Indeed, notice that the identity \eqref{diffx} follows directly from Lemma \ref{stirling} using the fact that $S_{(n)}^{(j)}=(-1)^{n-1}(n-1)!$.
\noindent Furthermore, by taking $\lambda = \log(k),\,1<k\in\mathbb{N}$, then 
\begin{align*}
D^{\log(k),(j)}_xx = \begin{cases}x ,\,\qquad\qquad \qquad \quad j=0\\
\frac{1}{\log(k)}(-1)^{j-1}(j-1)!,\,j\geq 1,
\end{cases}
\end{align*}
so that 
$$x = \sum_{j=1}^\infty\frac{(-1)^{(j-1)}}{j\log(k)}(k^x-1)^j.$$
For $k^x-1<1$. In particular, by taking  $x=-1$, we obtain
$$\sum_{j=1}^\infty\frac{1}{j}\left(\frac{k-1}{k}\right)^j=\log(k).$$ 
Similarly, taking $a(x)=x^2,\lambda=\log(2),x_0=0$, we obtain, by taking $x=1$ and $x=-1$, respectively:
$$\sum_{j=2}^\infty\frac{S^{(2)}_{(j)}}{j!} = \frac{\log(2)^2}{2}$$
$$\sum_{j=2}^\infty\frac{(-1)^jS^{(2)}_{(j)}}{2^jj!} = \frac{\log(2)^2}{2}.$$
More generally, we obtain:
$$\sum_{j=k}^\infty\frac{(-1)^jS^{(k)}_{(j)}}{j!} = \frac{\log(2)^k}{k!}$$
$$\sum_{j=k}^\infty\frac{(-1)^jS^{(k)}_{(j)}}{2^jj!} = (-1)^k\frac{\log(2)^k}{k!},$$
for all $k\in\mathbb{N}$.
\end{exmp}

Next, we develop the multivariable case. The commonly used approach for the usual Taylor series does not work well here, but the idea is also to use the 1 dimensional case to prove the general one. Notice it is necessary to consider functions defined in ``rectangles". 
First, consider the following multi-index notation:
\begin{definition}
Let $\gamma\in\mathbb{N}_0^n$ be a multi-index, $x = (x_1,...,x_n)\in\mathbb{R}^n.$ Define the partial differential operator:
$$D^{\lambda,(\gamma)}_x = D_{x_1}^{\lambda,(\gamma_1)}...D_{x_n}^{\lambda,(\gamma_n)},$$
where $D_{x_i}^{\lambda,(j)}$ is the differential operator defined in Definition \ref{def1}, applied on the variable $x_i$. Let also:
$$(e^{\lambda x}-1)^\gamma = (e^{\lambda x_1}-1)^{\gamma_1}...(e^{\lambda x_n}-1)^{\gamma_n}.$$
As usual, let $\|x\|_{\infty} = \max_{1\leq i \leq n}\{{|x_i|}\}$ and given another $\tilde{x}\in\mathbb{R}^n$, consider the set:
$$Q(x,\tilde{x}) = \{y\in\mathbb{R}^n | y_j\in \overline{[[x_j,\tilde{x}_j]]},1\leq j\leq n\},$$
that is, $y\in Q(x,\tilde{x})$ if and only if each of its coordinates lies between the corresponding $x$ and $\tilde{x}$ coordinate. 
\end{definition}
We are now able to state the multivariable version of Theorem \ref{teo1}:
\begin{theorem}
Let $n\in\mathbb{N}$, $a\in C^k(U_1\times ...\times U_n)$, where each $U_i\subset \mathbb{R}$ is open, $0\neq \lambda\in \mathbb{C},\,\tilde{x}\in\mathbb{R}^n$. Then
$$a(x) = \sum_{|\gamma|<N}\frac{1}{\gamma!}D^{\lambda,(\gamma)}a(\tilde{x})(e^{\lambda(x-\tilde{x})}-1)^\gamma + R_N(x,\tilde{x}),$$
where
\begin{align*}
|R_N(x,\tilde{x})| &\leq |\lambda|\left[\sum_{|\gamma| = N}\frac{N}{\gamma!}\sup_{y\in Q(x,\tilde{x})}|D^{\lambda,(\gamma)}a(y)|\right]\epsilon(\lambda,\|x-\tilde{x}\|_\infty)^{N-1} \|x-\tilde{x}\|_\infty,\\
\end{align*}
with
$$\epsilon(\lambda,r) \doteq \sup_{-r\leq z\leq r}|e^{\lambda z}-1|,$$
for each $N\leq k$.

\end{theorem}

\begin{proof}
We prove it by induction in $n\in\mathbb{N}$. The case $n=1$ follows trivially from Theorem \ref{teo1}. Let $n>1$ and suppose the theorem is true for $n-1\in\mathbb{N}$. For $x\in\mathbb{R}^n$, write $x=(x',x_n)$, where $x'= (x_1,...,x_{n-1})\in\mathbb{R}^{n-1},\,x_n\in\mathbb{R}$. For each $x_n\in U_n$ fixed, consider the function $a_{x_n}(x_1,...,x_{n-1}) = a(x_1,...,x_n)$, that is, $a_{x_n}(x') = a(x)$. Then such function is of class $C^k$ in $U_1\times...\times U_{n-1}\subset\mathbb{R}^{n-1}$, so by the inductive hypothesis:
\begin{equation*}
a_{x_n}(x') = \sum_{|\gamma|<N}\frac{1}{\gamma!}D^{\lambda,(\gamma)}a_{x_n}(\tilde{x}')(e^{\lambda(x'-\tilde{x}')}-1)^\gamma + R_{N,x_n}(x',\tilde{x}'),
\end{equation*}
where $\gamma\in\mathbb{N}_0^{n-1}$, for each $x_n\in U_n$. Moreover,
\begin{align*}
R_{N,x_n}(x',\tilde{x}') &=|\lambda|\left[\sum_{\substack{|\gamma| = N\\ \gamma\in\mathbb{N}_0^{n-1}}}\frac{N}{\gamma!}\sup_{y\in Q(x',\tilde{x}')}|D^{\lambda,(\gamma)}a_{x_n}(y)|\right]\\
&\times\sup_{-\|x'-\tilde{x}'\|_{\infty}\leq z\leq \|x'-\tilde{x}'\|_{\infty}}|e^{\lambda z}-1|^{N-1}\|x'-\tilde{x}'\|_\infty.
\end{align*}
Since $\|x'-\tilde{x}'\|_\infty\leq \|x-\tilde{x}\|_\infty$ and $Q(x',\tilde{x}')\times \{x_n\}\subset Q(x,\tilde{x})$, and
since for each $\gamma\in\mathbb{N}_0^{n-1},\,|\gamma|<N$, we can identify $\gamma = (\gamma,0)\in\mathbb{N}_0^n$, such that for every $\tau\in \mathbb{N}_0^n$. We also have $\tau = (\tau',\tau_n) = (\tau',0) +(0,\tau_n)$, $\tau'\in\mathbb{N}_0^{n-1}.$ Therefore $D^{\lambda,(\gamma)}a_{x_n}(\tilde{x}') = D^{\lambda,((\gamma,0))}a(\tilde{x}',x_n)$  that is 
\begin{equation}\label{eq1}
a(x',x_n) = \sum_{|\gamma|<N}\frac{1}{\gamma!}D^{\lambda,((\gamma,0))}a(\tilde{x}',x_n)(e^{\lambda(x-\tilde{x})}-1)^{(\gamma,0)} + R_{N,x_n}(x',\tilde{x}'),
\end{equation}
where
\begin{align}\label{errorbasecase}
R_{N,x_n}(x',\tilde{x}')&=|\lambda|\left[\sum_{\substack{|\gamma| = N\\ \gamma\in\mathbb{N}_0^{n-1}}}\frac{N}{\gamma!}\sup_{y\in Q(x,\tilde{x})}|D^{\lambda,((\gamma,0))}a(y)|\right]\nonumber  \\
&\times\sup_{-\|x-\tilde{x}\|_{\infty}\leq z\leq \|x-\tilde{x}\|_{\infty}}|e^{\lambda z}-1|^{N-1}\|x-\tilde{x}\|_\infty,
\end{align}
for each $x_n\in U_n$ (note that this expression does not depend on $x_n$). 
Now, for each $(\gamma,0)\in\mathbb{N}_0^n$, $|\gamma|<N$, we have that $D^{\lambda,((\gamma,0))}a(\tilde{x}',x_n)$ is of class $C^{k-|\gamma|}$ in $U_n$, so we may use Theorem \ref{teo1} to expand it in it's exponential Taylor series centered at $\tilde{x}'$ up to order $N-|\gamma|$, so that by Theorem \ref{teo1}, 
\begin{align*}
D^{\lambda,((\gamma,0))}a(\tilde{x}',x_n)&=\sum_{j=0}^{N-|\gamma|}\frac{1}{j!} D^{\lambda,((\gamma,j))}a(\tilde{x}',\tilde{x}_n)(e^{\lambda(x_n-\tilde{x}_n)}-1)^j+\\
&+R_{N-|\gamma|,\gamma}(x,\tilde{x}),
\end{align*} 
where
\begin{align*}
|R_{N-|\gamma|,\gamma}(x,\tilde{x})|&\leq |\lambda|\frac{N}{(N-|\gamma|)!}\sup_{y\in [[x_n,\tilde{x}_n]]}|D^{\lambda,(\gamma,N-|\gamma|)}a(\tilde{x}',y)|\\
&\times\sup_{-|\tilde{x}_n-x_n|\leq z\leq |\tilde{x}_n-x_n|}|e^{\lambda z}-1|^{N-1-|\gamma|}|x_n-\tilde{x}_n|\\
&\leq|\lambda|\frac{N}{(N-|\gamma|)!}\sup_{y\in Q(x,\tilde{x})}|D^{\lambda,(\gamma,N-|\gamma|)}a(y)|\\
&\times\sup_{-\|x-\tilde{x}\|_{\infty}\leq z\leq \|x-\tilde{x}\|_{\infty}}|e^{\lambda z}-1|^{N-1-|\gamma|}\|x-\tilde{x}\|_\infty.
\end{align*}
Now, substituting these expressions back in equation \eqref{eq1}, we obtain
\begin{align*}
a(x) = a(x',x_n) &= \sum_{\substack{|\gamma|<N\\\gamma\in\mathbb{N}_0^{n-1}}}\frac{1}{\gamma!}\sum_{j=0}^{N-|\gamma|-1}\frac{1}{j!} D^{\lambda,((\gamma,j))}a(\tilde{x}',\tilde{x}_n)(e^{\lambda(x-\tilde{x})}-1)^{(\gamma,0)}(e^{\lambda(x-\tilde{x})}-1)^{(0,j)}\\
& + \sum_{\substack{|\gamma|<N\\\gamma\in\mathbb{N}_0^{n-1}}}\frac{1}{\gamma!}R_{N-|\gamma|,\gamma}(x,\tilde{x})(e^{\lambda(x-\tilde{x})}-1)^{(\gamma,0)} +  R_{N,x_n}(x',\tilde{x}').
\end{align*}
Note that the first term may be rewritten as
$$\sum_{\substack{|\gamma|<N\\ \gamma\in\mathbb{N}_0^n}}\frac{1}{\gamma!}D^{\lambda,(\gamma)}a(\tilde{x})(e^{\lambda(x-\tilde{x})}-1)^\gamma.$$
Moreover, since
$$\sup_{-k\leq z\leq k} |e^{\lambda z}-1|^{r} |e^{\lambda z}-1|^{s}= \sup_{-k\leq z\leq k} |e^{\lambda z}-1|^{r+s}.$$
and $$\{(\gamma,N-|\gamma|);\gamma\in\mathbb{N}_0^{n-1},\,|\gamma|<N\}\cup\{(\gamma,0);|\gamma|=N\} = \{\gamma\in\mathbb{N}_0^n;|\gamma|=N\},$$
setting the remainder as 
$$R_N(x,\tilde{x}) = \sum_{\substack{|\gamma|<N\\\gamma\in\mathbb{N}_0^{n-1}}}\frac{1}{\gamma!}R_{N-|\gamma|,\gamma}(x,\tilde{x})(e^{\lambda(x-\tilde{x})}-1)^{(\gamma,0)} +  R_{N,x_n}(x',\tilde{x}'),$$
and using the estimate \eqref{errorbasecase}, we have that: 
\begin{align*}
|R_N(x,\tilde{x})| &\leq|\lambda|\sum_{\substack{|\gamma|<N\\ \gamma\in\mathbb{N}_0^{n-1}}}\frac{1}{\gamma!}\frac{N}{(N-|\gamma|)!}\sup_{y\in Q(x,\tilde{x})}|D^{\lambda,(\gamma,N-|\gamma|)}a(y)|\times\\
&\times\sup_{-\|x-\tilde{x}\|_{\infty}\leq z\leq \|x-\tilde{x}\|_{\infty}}|e^{\lambda z}-1|^{N-1-|\gamma|}\|e^{\lambda z}-1|^{|\gamma|}\|x-\tilde{x}\|_\infty+\\
&+\left[\sum_{\substack{|\gamma| = N\\ \gamma\in\mathbb{N}_0^{n-1}}}\frac{N}{\gamma!}\sup_{y\in Q(x,\tilde{x})}|D^{\lambda,((\gamma,0))}a(y)|\right]\times\\
&\times\sup_{-\|x-\tilde{x}\|_{\infty}\leq z\leq \|x-\tilde{x}\|_{\infty}}|e^{\lambda z}-1|^{N-1}\|x-\tilde{x}\|_\infty\\
&= |\lambda|\left[\sum_{\substack{|\gamma|=N\\ \gamma\in\mathbb{N}_0^{n}}}\frac{N}{\gamma!}\sup_{y\in Q(x,\tilde{x})}|D^{\lambda,(\gamma)}a(y)|\right]\epsilon(\lambda,\|x-\tilde{x}\|_\infty)^{N-1} \|x-\tilde{x}\|_\infty,\\
\end{align*}
where
$$\epsilon(\lambda,r) = \sup_{-r\leq z\leq r}|e^{\lambda z}-1|.$$
\end{proof}

\begin{corollary}
Let $a\in C^{\infty}(U_1\times...\times U_n)$, $\tilde{x}\in U_1\times...\times U_n\subset \mathbb{R}^n$, and suppose that there exists $V\subset\mathbb{R}^n$ neighbourhood of $\tilde{
x}$, $A\geq0$ such that 
$$ \sup_{y\in V}|D^{\lambda,(\gamma)}a(y)|\leq A \gamma!,$$
for every multi-index $\gamma$. Then there exists $\delta_{\lambda}>0$  such that 
$$a(x) = \sum_{\gamma\in\mathbb{N}_0^n}\frac{1}{\gamma!}D^{\lambda,(\gamma)}a(\tilde{x})(e^{\lambda(x-\tilde{x})}-1)^\gamma$$ 
for every $x\in\mathbb{R}^n$ with $\|x-\tilde{x}\|_\infty<\delta_{\lambda}$.
\end{corollary}
\begin{proof}
Note that in the notation of the previous theorem, it is enough to show that $|R_{N}(x,\tilde{x})|\to 0$ as $N\to \infty$. Also note that
$$ \#\{\gamma\in\mathbb{N}^n_0 ; |\gamma| = N\} = {N+n-1\choose N} = \frac{N^{n-1}}{(n-1)! }+O(N^{n-2}).$$
Therefore, using the fact that $\mathbb{R}\to\mathbb{C}:z\mapsto e^{\lambda z}-1$ is continuous at $z=0$, and taking $\delta_\lambda>0$ small enough so that $|e^{\lambda z}-1|\leq \alpha<1$ for $|z|<\delta_\lambda$ and $\{x;\|x-\tilde{x}\|_\infty<\delta_\lambda\}\subset V$, for $\|x-\tilde{x}\|_{\infty}<\delta_\lambda$ we have that
\begin{align*}
|R_N(x,\tilde{x})| &\leq |\lambda|\left[\sum_{|\gamma| = N}\frac{N}{\gamma!}\sup_{y\in Q(x,\tilde{x})}|D^{\lambda,(\gamma)}a(y)|\right]\epsilon(\lambda,\|x-\tilde{x}\|_\infty)^{N-1} \|x-\tilde{x}\|_\infty\\
&\leq|\lambda|\left(\frac{N^{n-1}}{(n-1)!}+O(N^{n-2})\right)\frac{N}{\gamma!}A\gamma!\alpha^{N-1}\delta_{\lambda}\\
&=A\delta_{\lambda}|\lambda|\left(\frac{N^{n}}{(n-1)!}+O(N^{n-1})\right)\alpha^{N-1}\to 0,
\end{align*}
as $N\to\infty$, so the result follows.
\end{proof}

\begin{exmp}
Let $a\in C^{\infty}(\mathbb{R}^2)$ be given by $a(x,y) = \cos(2\pi(x+y))$. Induction and a simple calculation show that
\begin{align*}
D^{(2\pi i),(j,k)}\cos(2\pi(x+y) = \frac{(-1)^{j+k}}{2}j!k!e^{-2\pi i(x+y)},
\end{align*}
if $j\geq 2$ or if $k\geq 2$. Calculating the other low order terms individually, taking the exponential taylor expansion of $a$ at $\tilde{x} = (0,0)$, we have that
\begin{align*}
\cos(2\pi(x+y)) &= 1+ (e^{2\pi ix}-1)(e^{2\pi iy}-1)\\
&+ \sum_{j=2}^\infty\sum_{k=0}^\infty\frac{(-1)^{j+k}}{2}(e^{2\pi ix}-1)^j(e^{2\pi iy}-1)^k\\
&+ \sum_{j=0}^\infty\sum_{k=2}^ \infty\frac{(-1)^{j+k}}{2}(e^{2\pi ix}-1)^j(e^{2\pi iy}-1)^k,
\end{align*}
for $|x|<\frac{1}{6}$ and $|y|<\frac{1}{6}$. Notice that by taking $x=y=\frac{1}{12}$, we are able to calculate $\cos(2\pi/6)$, which was not possible using the series expansion centered at $0$ of $\cos(2\pi x)$, that is, we manage to ``avoid the singularity" at that point.
\end{exmp}

\end{document}